\documentclass[12pt]{amsart} 

\usepackage{fullpage}
\usepackage{url,amssymb,amsmath,amsfonts,amsthm,mathrsfs}

\usepackage[all]{xy} \SelectTips{cm}{10}

\usepackage{lipsum} 

\usepackage{color}

\definecolor{webcolor}{rgb}{0.8,0,0.2}
\definecolor{webbrown}{rgb}{.6,0,0}
\usepackage[
        colorlinks,
        linkcolor=webbrown,  filecolor=webcolor,  citecolor=webbrown, 
        backref,
        pdfauthor={David Zywina}, 
       pdftitle={},
]{hyperref}
\usepackage[alphabetic,backrefs,lite]{amsrefs} 

\numberwithin{equation}{section}

\DeclareFontEncoding{OT2}{}{} 
\newcommand{\textcyr}[1]{%
 {\fontencoding{OT2}\fontfamily{cmr}\fontseries{m}\fontshape{n}\selectfont #1}}

\newcommand{\Sha}{{\mbox{\textcyr{Sh}}}}

\newcommand{\FF}{\mathbb F}

\newcommand{\QQ}{\mathbb Q}
\newcommand{\RR}{\mathbb R}

\newcommand{\ZZ}{\mathbb Z}

\newcommand{\calP}{\mathcal P}

\newcommand{\calE}{\mathcal E}

\def\Sel{\operatorname{Sel}} 

\def\Gal{\operatorname{Gal}}
\def\ord{\operatorname{ord}}

\newcommand{\tors}{{\operatorname{tors}}}

\newcommand{\legendre}[2]{\big(\tfrac{#1}{#2}\big)}

\newcommand{\defi}[1]{\textsf{#1}} 

\newcommand\blank[1]{}

\def\bbar#1{\setbox0=\hbox{$#1$}\dimen0=.2\ht0 \kern\dimen0 
\overline{\kern-\dimen0 #1}}
\newcommand{\Qbar}{{\overline{\mathbb Q}}}

\newtheorem{thm}{Theorem}[section]
\newtheorem{lemma}[thm]{Lemma}

\theoremstyle{definition}

\theoremstyle{remark}

\newenvironment{romanenum}{\hfill \begin{enumerate} }{\end{enumerate}}

\begin{document}

\title{There are infinitely many elliptic curves over the rationals of rank $2$}

\author{David Zywina}
\address{Department of Mathematics, Cornell University, Ithaca, NY 14853, USA}
\email{zywina@math.cornell.edu}

\begin{abstract}
We show that there are infinitely many elliptic curves $E/\QQ$, up to isomorphism over $\Qbar$, for which the finitely generated group $E(\QQ)$ has rank exactly $2$.    Our elliptic curves are given by  explicit models and their rank is shown to be $2$ via a $2$-descent.  That there are infinitely many such elliptic curves makes use of a theorem of Tao and Ziegler.
\end{abstract}

\maketitle

\section{Introduction}

For an elliptic curve $E$ over $\QQ$, the abelian group $E(\QQ)$ consisting of the rational points of $E$ is finitely generated.   The possible torsion subgroups of $E(\QQ)$ that can occur have been classified by Mazur \cite[Theorem~8]{MR488287}.
On the other hand, the \defi{rank} of $E$, i.e., the rank of $E(\QQ)$, is a more mysterious invariant.     For each integer $0\leq r \leq 1$, it is known that there are infinitely many elliptic curves over $\QQ$ of rank $r$ (for example, see \cite[Corollary X.6.2.1]{Silverman} and \cite{MR870738}).   Loosely, we expect ``most'' elliptic curves over $\QQ$ to have rank $0$ or $1$.

There are in fact infinitely many elliptic curve over $\QQ$ that have rank \emph{at least} $2$.  Indeed, if one takes a nonisotrivial elliptic curve $\calE$ over the function field $\QQ(T)$ for which $\calE(\QQ(T))$ has rank $2$, then specialization at all but finitely many $t\in \QQ$ will produce an elliptic curve over $\QQ$ of rank \emph{at least} $2$, cf.~\cite{MR703488}.
Our main result gives what seems to be the first integer $r\geq 2$ that we can confirm is the rank of infinitely many elliptic curves over $\QQ$.

\begin{thm}\label{T:MAIN}
There are infinitely many elliptic curves over $\QQ$, up to isomorphism over $\Qbar$, of rank $2$.
\end{thm}

We will prove Theorem~\ref{T:MAIN} by showing that a specific class of elliptic curves over $\QQ$ has rank $2$.

\begin{thm}\label{T:main examples}
Let $m$ and $n$ be any natural numbers for which $m$, $m+16n^2$ and $m+25n^2$ are primes congruent to $11$ modulo $24$.  Let $E$ be the elliptic curve over $\QQ$ defined by the equation 
\[
y^2 = x^3-5(m+16n^2)x^2+4(m+16n^2)(m+25n^2)x.
\]
Then $E(\QQ) \cong \ZZ/2\ZZ \times \ZZ^2$.
\end{thm}

We will make use of the \emph{Polynomial Szemer\'edi theorem for primes} due to Tao and Ziegler \cite{TZ} to guarantee that there are infinitely many pairs $(m,n)$ of natural numbers for which $m$, $m+16n^2$ and $m+25n^2$ are primes congruent to $11$ modulo $24$.  This is the source of the infiniteness in our proof of Theorem~\ref{T:MAIN}.

With $E/\QQ$ as in Theorem~\ref{T:main examples}, we have the following points in $E(\QQ)$:
\[
P_0=(0,0),\quad P_1=(m+16n^2, 6n(m+16n^2)), \quad P_2=(36n^2, 12n(m-2n^2)).
\]
The point $P_0$ has order $2$ and our proof will show that $P_1$ and $P_2$ generate a free abelian group of rank $2$.  Moreover, we will see that the points $P_0$, $P_1$ and $P_2$ generate the group $E(\QQ)/2E(\QQ)\cong (\ZZ/2\ZZ)^3$.

 The bad primes of $E$ are precisely $2$, $3$, $m$, $m+16n^2$ and $m+25n^2$.   To determine $E(\QQ)/2E(\QQ)$ we will perform a descent using an isogeny of degree $2$.   Our complete knowledge of the bad primes will allow us to directly compute the relevant Selmer groups.  For the elliptic curves $E/\QQ$ we are considering, we are fortunate to always have $\Sha(E/\QQ)[2]=0$.

\subsection{Aside: remarks concerning our curves}
The proof of Theorem~\ref{T:main examples} is a straightforward and pleasant descent computation.   The serious work in the theorem was finding these elliptic curves in the first place! Let us briefly make a few observations of the important properties that these curves have.  This will hopefully serve as motivation.    This material will not be used elsewhere in the paper.

Let $\calE$ be the elliptic curve over the function field $\QQ(T)$ defined by the Weierstrass equation
\[
y^2=x^3-5(T+16)x^2 + 4(T+16)(T+25)x;
\]
it has discriminant $2^83^2T(T+16)^3(T+25)^2$.    The two points $(T+16,6(T+16))$ and $(36,12(T-2))$ in $E(\QQ(T))$ have infinite order and are independent.  Moreover, $\calE(\QQ(T))$ has rank $2$.  The singular fibers of $\calE$ are at $0$, $-16$, $-25$ and $\infty$ with Kodaira symbols $\operatorname{I}_1$,  $\operatorname{III}$, $\operatorname{I}_2$ and $\operatorname{I}_0^*$, respectively.

Since $\calE$ is nonisotrivial, it produces infinitely many elliptic curves over $\QQ$ of rank at least $2$.   For every  $t\in \QQ-\{0,-16,-25\}$, let $\calE_t$ be the elliptic curve over $\QQ$ obtained by replacing $T$ by $t$.  From a theorem of Silverman \cite{MR703488}, $\calE_t$ has rank at least $2$ for all but finitely many $t$.   The challenge for us is that the ranks could be larger.    

Now consider  $t=m/n^2$ with $m$ and $n$ fixed natural numbers that are relative prime.    The curve $\calE_t$ is isomorphic to the elliptic curve $E/\QQ$ defined by the equation in Theorem~\ref{T:main examples}; the isomorphism is simply scaling the variables by suitable powers of $n$.  The discriminant for the Weierstrass model of $E$ is 
\[
2^8 \cdot  3^2\cdot m \cdot (m+16n^2)^3\cdot (m+25n^2)^2
\]
which gives constraints on the possible bad primes of $E$.    We chose $t$ to have square denominator since otherwise the curve $\calE_t$ might also have bad reduction at primes dividing the denominator of $t$  (this was arranged by having the singular fiber of $\calE$ at $\infty$ to have Kodaira symbol $\operatorname{I}_0^*$).

We now assume that $m$ and $n$ are chosen so that $m$, $m+16n^2$ and $m+25n^2$ are all primes with $m>5$.   Using this, we find that $\calP:=\{2,3,m,m+16n^2,m+25n^2\}$ is precisely the set of bad primes for $E$.  Note that in order to apply the theorem of Tao and Ziegler and obtain infinitely many such pairs $(m,n)$, it was important for the singular fibers of $\calE$ to occur only at integers and $\infty$.

Let $W(E) \in \{\pm 1\}$ be the global root number of $E/\QQ$.   The \emph{parity conjecture} predicts that the rank of $E$ is even if and only if $W(E)=1$.   Since we are trying to find elliptic curves of rank $2$, we should thus restrict to pairs $(m,n)$ for which we know that $W(E)=1$.  We have $W(E)=-\prod_{p\in \calP} W_p$, where $W_p$ is the local root number of $E$ at $p$.   After some local root number computations, we find that
\[
W(E)= W_2\cdot  W_3 \cdot (-1)^{\tfrac{m+1}{2}}.
\]
The assumption $m\equiv 11 \pmod{24}$ in Theorem~\ref{T:main examples} implies that $W_2=W_3=1$ and hence $W(E)=1$.   If instead we were to take $m\equiv 5 \pmod{24}$, then we would always have $W(E)=-1$.    In the case $m\equiv 5 \pmod{24}$, a similar proof to that of this paper will show that there are infinitely many elliptic curves over $\QQ$ with root number $-1$ and rank $2$ or $3$.  This can be used to show that, assuming the parity conjecture, there are infinitely many elliptic curves over $\QQ$ of rank $3$.   Similarly, there were earlier results showing that, under the parity conjecture, there are infinitely many elliptic curves over $\QQ$ of rank $2$, cf.~\cite{MR3492814, MR3960281}.

Our elliptic curve $E$ has a rational $2$-torsion point and we can thus perform a descent by using an isogeny $\phi\colon E\to E'$ of degree $2$.   In our proof of Theorem~\ref{T:main examples}, we will compute the Selmer groups $\Sel_\phi(E/\QQ)$ and $\Sel_{\hat\phi}(E'/\QQ)$, where $\hat\phi\colon E'\to E$ is the dual isogeny of $\phi$.    The natural homomorphisms
\begin{align} \label{E:Sel intro}
E'(\QQ)/\phi(E(\QQ)) \hookrightarrow \Sel_\phi(E/\QQ) \quad \text{ and }\quad E(\QQ)/\hat\phi(E'(\QQ)) \hookrightarrow \Sel_{\hat\phi}(E'/\QQ)
\end{align}
will be shown to be isomorphisms and from this we will compute $E(\QQ)/2E(\QQ)$.

The congruence conditions on $m$, $m+16n^2$ and $m+25n^2$ in Theorem~\ref{T:main examples} will be used during the Selmer group computations.    Without these imposed congruences one can find examples where $E/\QQ$ has root number $1$ and the homomorphisms (\ref{E:Sel intro}) are not both isomorphisms; in such cases we will have $\Sha(E/\QQ)[2]\neq 0$.

This article is a sequel to \cite{Zyw25} where we find an example of a nonisotrivial elliptic curve $\calE$ over $\QQ(T)$ for which $\calE_t$ is an elliptic curve over $\QQ$ of rank $0$ for infinitely many $t\in \QQ$.   The computations here are much simpler that in \cite{Zyw25} in part because the extra rational points give rise to elements in our Selmer groups.

\section{Descent via two-isogeny} \label{S:background}

In this section, we recall basic definitions and results concerning descent via a two-isogeny.  See \cite[\S X.4]{Silverman}, and especially  \cite[\S X.4 Example 4.8]{Silverman}, for the relevant formulae.

We start with an elliptic curve $E/\QQ$ defined by a Weierstrass equation $y^2=x(x^2+ax+b)$, where $a$ and $b$ are integers.  With $a':=-2a$ and $b':=a^2-4b$, we let $E'$ be the elliptic curve over $\QQ$ given by the model $y^2=x(x^2+a'x+b')$.   There is an isogeny $\phi\colon E\to E'$ given by 
\[
\phi(x,y)=\Big(\frac{y^2}{x^2}, \frac{y(b-x^2)}{x^2} \Big).
\]
whose kernel $E[\phi]$ is cyclic of order $2$ and generated by $(0,0)$.   Let $\hat\phi\colon E'\to E$ be the dual isogeny of $\phi$; its kernel $E'[\hat\phi]$ is generated by the $2$-torsion point $(0,0)$ of $E'$.  We have
\[
\hat\phi(x,y)=\Big(\frac{y^2}{4x^2}, \frac{y(b'-x^2)}{8x^2} \Big).
\]

For each $d \in \QQ^\times$, let $C_d$ be the smooth projective curve over $\QQ$ defined by the affine equation
\begin{align} \label{E:Cd background}
y^2=d x^4+a'x^2 + b'/d.
\end{align}

 Set $\Gal_\QQ:=\Gal(\Qbar/\QQ)$.  Starting with the short exact sequence $0\to E[\phi]\to E \xrightarrow{\phi} E'\to 0$ and taking Galois cohomology yields an exact sequence
\[
0 \to E(\QQ)[\phi] \to E(\QQ)\xrightarrow{\phi} E'(\QQ) \xrightarrow{\delta} H^1(\Gal_\QQ, E[\phi]).
\]
Since $E[\phi]$ and $\{\pm 1\}$ are isomorphic $\Gal_\QQ$-modules, we have a natural isomorphism
\begin{align} \label{E:H1 isom}
H^1(\Gal_\QQ, E[\phi])\xrightarrow{\sim} H^1(\Gal_\QQ, \{\pm 1\}) \xrightarrow{\sim} \QQ^\times/(\QQ^\times)^2,
\end{align}
where the last isomorphism is using that each extension of $\QQ$ of degree at most $2$ is obtained by adjoining the square root from a unique square class.  Using the isomorphism (\ref{E:H1 isom}), we may view $\delta$ as a homomorphism 
\[
\delta\colon E'(\QQ)\to  \QQ^\times/(\QQ^\times)^2.
\]   
For any point $(x,y)\in E'(\QQ)-\{0,(0,0)\}$, we have 
\[
\delta((x,y))=x\cdot (\QQ^\times)^2.
\]  
We also have $\delta(0)=1$ and $\delta((0,0))=b'\cdot (\QQ^\times)^2$.

Let $\Sel_\phi(E/\QQ) \subseteq H^1(\Gal_\QQ, E[\phi])$ be the \defi{$\phi$-Selmer group} of $E$.  Using (\ref{E:H1 isom}), we can identify $\Sel_\phi(E/\QQ)$ with a subgroup of $ \QQ^\times/(\QQ^\times)^2$.  In fact, we have
\[
\Sel_\phi(E/\QQ)=\{ d\in \QQ^\times/(\QQ^\times)^2: C_d(\QQ_v)\neq \emptyset \text{ for all places $v$ of $\QQ$}\}
\]
which we will use as our working definition of the $\phi$-Selmer group.   

The importance of $\Sel_\phi(E/\QQ)$ is that it is a finite computable group that contains the image of $\delta$.     In particular, $\delta$ gives rise to an injective homomorphism
\[
E'(\QQ)/\phi(E(\QQ)) \hookrightarrow \Sel_\phi(E/\QQ).
\]

For a prime $p$, we will denote by $\ord_p\colon \QQ_p \to \ZZ \cup \{\infty\}$ to be the $p$-adic discrete valuation normalized so that $\ord_p(p)=1$.

\begin{lemma} \label{L:easy d}
Let $d$ be a squarefree integer representing a square class in $\Sel_\phi(E/\QQ)$.  Then $d$ divides $b'$.
\end{lemma}
\begin{proof}
Suppose that there is a prime $p|d$ that does not divide $b'$.  By our assumptions on $d$, we will have $C_d(\QQ_p)\neq \emptyset$.   The integer $d$ is not a square in $\QQ_p$, since it is only divisible by $p$ once, and hence (\ref{E:Cd background}) has no $\QQ_p$-points at infinity.   Fix a point $(x,y)\in \QQ_p^2$ satisfying (\ref{E:Cd background}).   

Suppose that $x\in \ZZ_p$.   We have $dx^4+a'x^2\in \ZZ_p$ and $\ord_p(b'/d)=-1$, where we have used that $a'$ and $b'$ are integers and that $p\nmid b'$.  Therefore, $\ord_p(y^2)=\ord_p(b'/d)=-1$ which contradicts that $\ord_p(y^2)=2\ord_p(y)$ is an even integer.  We thus have $x\notin \ZZ_p$.

Define $e:=-\ord_p(x)\geq 1$. Comparing $p$-adic valuations, we find that $\ord_p(y^2)=\ord_p(dx^4)=-4e+1$ which again contradicts that $\ord_p(y^2)$ is even.
\end{proof}

Similarly, we have a homomorphism $\delta'\colon E(\QQ)\to \QQ^\times/(\QQ^\times)^2$ that has kernel $\hat\phi(E'(\QQ))$ and the image of $\delta'$ lies in the $\hat\phi$-Selmer group $\Sel_{\hat\phi}(E'/\QQ) \subseteq \QQ^\times/(\QQ^\times)^2$.

There is an exact sequence
\[
0 \to E'(\QQ)[\hat\phi]/\phi(E(\QQ)[2])\to E'(\QQ)/\phi(E(\QQ)) \xrightarrow{\hat\phi} E(\QQ)/2E(\QQ)\to E(\QQ)/\hat\phi(E'(\QQ)) \to 0.
\]
Once we know the group $E(\QQ)/2E(\QQ)$, it will be straightforward to compute the rank of $E(\QQ)$.   

\section{Proof of Theorem~\ref{T:main examples}}

Let $E/\QQ$ be an elliptic curve from Theorem~\ref{T:main examples}; it is given by an equation
\[
y^2 = x\big(x^2-5(m+16n^2)x+4(m+16n^2)(m+25n^2)\big),
\]
where $m$ and $n$ are natural numbers for which $m$, $m+16n^2$ and $m+25n^2$ are primes congruent to $11$ modulo $24$.    The discriminant of this Weierstrass model is 
\[
\Delta= 2^8 \cdot  3^2\cdot m \cdot (m+16n^2)^3\cdot (m+25n^2)^2.
\]
The above Weierstrass model of $E/\QQ$ is minimal since the exponents of the primes dividing $\Delta$ are all strictly less than $12$.   Define the points $P_0=(0,0)$ and $P_1=(m+16n^2, 6n(m+16n^2))$ of $E(\QQ)$.

We follow the notation of \S\ref{S:background}; we have $a=-5(m+16n^2)$, $b=4(m+16n^2)(m+25n^2)$, $a'=-2a=10(m+16n^2)$ and $b'=a^2-4b=9m(m+16n^2)$.  Define the elliptic curve $E'$ over $\QQ$ by 
\[
y^2=x(x^2+a'x+b')=x\big(x^2+10(m+16n^2)x+9m(m+16n^2)\big).
\]
Define the points $Q_0=(0,0)$ and $Q_1=(-(m+16n^2),12n(m+16n^2))$ of $E'(\QQ)$.

As in \S\ref{S:background}, we have an explicit isogeny $\phi\colon E\to E'$ of degree $2$ with its dual isogeny $\hat\phi$.   Along with $P_0$ and $P_1$, we also define a third point of $E(\QQ)$:
\[
P_2:=\hat\phi(Q_1)=(36n^2, 12n(m-2n^2)).
\]

\subsection{Selmer group computations}
We will now show that the Selmer groups $\Sel_\phi(E/\QQ)$ and $\Sel_{\hat\phi}(E/\QQ)$ both have cardinality at most $4$.   To ease notation, we will denote an element of $\QQ^\times/(\QQ^\times)^2$ by the unique squarefree integer it contains. 

\begin{lemma} \label{L:Selmer 1}
We have $\Sel_\phi(E/\QQ) \subseteq \{1,-m, -(m+16n^2), m(m+16n^2)\}$.
\end{lemma}
\begin{proof}
Take any squarefree integer $d$ representing an element of $\Sel_\phi(E/\QQ)$.   Let $C_d$ be the smooth projective curve over $\QQ$ defined by the affine equation
\begin{align} \label{E:Selmer 1a}
y^2=d x^4+10(m+16n^2)x^2 + 9m(m+16n^2)/d.
\end{align}
Multiplying by $d$ and completing the square gives
\begin{align} \label{E:Selmer 1b}
dy^2=(dx^2+5(m+16n^2) )^2-16(m+16n^2)(m+25n^2).
\end{align}
We have $C_d(\QQ_p) \neq\emptyset$ for all primes $p$ by our choice of $d$.  

Suppose that $d\equiv 3 \pmod{4}$.   The model (\ref{E:Selmer 1a}) has no $\QQ_2$-point at infinity since $d$ is not a square in $\QQ_2$.  Choose a pair $(x,y)\in \QQ_2^2$ satisfying (\ref{E:Selmer 1a}).   If $x\in 2\ZZ_2$, then $y\in \ZZ_2$ and $y^2\equiv 9m(m+16n^2)/d \equiv m^2/3 \equiv 3 \pmod{4}$.  We thus have $x\notin 2\ZZ_2$ since $3$ is not a square modulo $4$.  Now suppose that $x\in \ZZ_2^\times$.   By (\ref{E:Selmer 1b}), we have $dy^2=z^2-16(m+16n^2)(m+25n^2)$ with $z:=dx^2+5(m+16n^2) \in \ZZ_2$ and $y\in \ZZ_2$.  We have $z^2\equiv dy^2\equiv 3 y^2 \pmod{16}$ from which we deduce that $z\equiv 0 \pmod{4}$.  Therefore, $dx^2 \equiv -5(m+16n^2) \equiv 1 \pmod{4}$, where we have used that the prime $m+16n^2$ is $3$ modulo $8$.  So $x^2\equiv 3 \pmod{4}$ which is a contradiction.  Therefore, $x\notin \ZZ_2$.   Define $e:=-\ord_2(x) \geq 1$.   We have $\ord_2(d x^4)=-4e$, $\ord_2(10(m+16n^2)x^2)=-2e+1$ and $\ord_2(9m(m+16n^2)/d)=1$.  From (\ref{E:Selmer 1a}), we deduce that $\ord_2(y^2)=-4e$ and hence $\ord_2(y)=-2e$.   Multiplying (\ref{E:Selmer 1a}) by $2^{4e}$ gives $(2^{2e} y)^2= d(2^e x)^4+2^{2e}10(m+16n^2)(2^ex)^2 + 2^{4e}9m(m+16n^2)/d$.  Using that $2^{2e}y,2^e x \in \ZZ_2^\times$ and reducing modulo $4$, we find that $d$ is a square modulo $4$ which again is a contradiction.    We conclude that $d\equiv 1 \pmod{4}$.

By Lemma~\ref{L:easy d}, $d$ divides $9m(m+16n^2)$.  Therefore, $d$ lies in the subgroup of $\QQ^\times/(\QQ^\times)^2$ generated by $-1$, $3$, $m$ and $m+16n^2$. Since $d\equiv 1 \mod{4}$ and all of the values $-1$, $3$, $m$ and $m+16n^2$ are congruent to $3$ modulo $4$, we have
\begin{align} \label{E:Selmer1 set}
d\in \{1, -3, -m, -(m+16n^2), 3m, 3(m+16n^2), m(m+16n^2),-3m(m+16n^2)\}.
\end{align}

Now suppose that $d$ is not a square modulo the prime $p:=m+25n^2$.    The model (\ref{E:Selmer 1a}) has no $\QQ_p$-point at infinity since $d$ is not a square in $\QQ_p$. Choose a pair $(x,y)\in \QQ_p^2$ satisfying (\ref{E:Selmer 1a}).   First suppose that $x\in \ZZ_p$.  By (\ref{E:Selmer 1b}), we have $dy^2=z^2-16(m+16n^2)p$ with $z\in \ZZ_p$.   If $z\in \ZZ_p^\times$, then $dy^2\equiv z^2 \pmod{p}$ with $z\not\equiv 0 \pmod{p}$ which contradicts that $d$ is not a square modulo $p$.  If $z\in p\ZZ_p$, then $dy^2 \equiv 16(m+16n^2)p \pmod{p^2}$ which implies that the even integer $2\ord_p(y)=\ord_p(y^2)$ is equal to $1$.  So we must have $x\notin \ZZ_p$.    Define $e:=-\ord_p(x) \geq 1$.   We have $\ord_p(d x^4)=-4e$, $\ord_p(10(m+16n^2)x^2)=-2e$ and $\ord_p(9m(m+16n^2)/d)=0$.  From (\ref{E:Selmer 1a}), we deduce that $\ord_p(y^2)=-4e$ and hence $\ord_p(y)=-2e$.   Multiplying (\ref{E:Selmer 1a}) by $p^{4e}$ gives $(p^{2e} y)^2= d(p^e x)^4+p^{2e}10(m+16n^2)(p^ex)^2 + p^{4e}9m(m+16n^2)/d$.  Using that $p^{2e}y,p^e x \in \ZZ_p^\times$ and reducing modulo $p$, we find that $d$ is a square modulo $p$ which again is a contradiction.    We conclude that $d$ is a square modulo $p$.

We now need to work out which of the integers in the set (\ref{E:Selmer1 set}) are squares modulo $p=m+25n^2$.   Since $p\equiv 3 \pmod{4}$ and $p\equiv 2 \pmod{3}$, we have  $\legendre{-1}{p}=-1$ and $\legendre{3}{p}=-\legendre{p}{3}=-\legendre{2}{3}=1$. We also have $\legendre{m+16n^2}{p}= \legendre{-25n^2+16n^2}{p} = \legendre{-9n^2}{p}=\legendre{-1}{p}$, where we have used $m\equiv -25n^2 \pmod{p}$ and $p\nmid n$ (since $n<p$).   Applying these Legendre symbols to the integers in the set (\ref{E:Selmer1 set}) and using that $d$ is a square modulo $p$, we conclude that $d$ lies in $\{1,  -m, -(m+16n^2), m(m+16n^2)\}$.
\end{proof}

\begin{lemma} \label{L:Selmer 2}
We have $\Sel_{\hat\phi}(E'/\QQ) \subseteq \{1, m+16n^2, m+25n^2, (m+16n^2)(m+25n^2)\}$.
\end{lemma}
\begin{proof}
Take any squarefree integer $d$ representing an element of $\Sel_\phi(E'/\QQ)$.   Let $C'_d$ be the smooth projective curve over $\QQ$ defined by the affine equation
\begin{align} \label{E:Selmer 2a}
y^2=d x^4 -5(m+16n^2)x^2+4(m+16n^2)(m+25n^2)/d.
\end{align}
Multiplying by $d$ and completing the square gives
\begin{align} \label{E:Selmer 2b}
dy^2=(dx^2-\tfrac{5}{2}(m+16n^2))^2-\tfrac{9}{4}m(m+16n^2).
\end{align}

Suppose that $d<0$.  Since $d$ is not a square in $\RR$, the model (\ref{E:Selmer 2a}) does not have a real point at infinity.  We have $C_d'(\RR)\neq \emptyset$ by our choice of $d$, so there is a pair $(x,y)\in \RR$ satisfying (\ref{E:Selmer 2a}).  Using that $d<0$, we have
  \begin{align*}
  (dx^2-\tfrac{5}{2}(m+16n^2))^2-\tfrac{9}{4}m(m+16n^2) & \geq(\tfrac{5}{2}(m+16n^2))^2-\tfrac{9}{4}m(m+16n^2) \\ &= 4(m+16n^2)(m+25n^2)>0.
  \end{align*}
From (\ref{E:Selmer 2b}), we obtain $dy^2>0$ which contradicts that $d<0$.  Therefore, $d>0$.

Now suppose that $d$ is not a square modulo the prime $p:=m$.    The model (\ref{E:Selmer 2a}) has no $\QQ_p$-point at infinity since $d$ is not a square in $\QQ_p$. Choose a pair $(x,y)\in \QQ_p^2$ satisfying (\ref{E:Selmer 2a}).   First suppose that $x\in \ZZ_p$.   By (\ref{E:Selmer 2b}), we have $dy^2=z^2-\tfrac{9}{4}(m+16n^2)p$ with $z\in \ZZ_p$.   If $z\in \ZZ_p^\times$, then $dy^2\equiv z^2 \pmod{p}$ with $z\not\equiv 0 \pmod{p}$ which contradicts that $d$ is not a square modulo $p$.  If $z\in p\ZZ_p$, then $dy^2 \equiv -\tfrac{9}{4}(m+16n^2)p \pmod{p^2}$ which implies that the even integer $2\ord_p(y)=\ord_p(y^2)$ is equal to $1$.  So we must have $x\notin \ZZ_p$.    Define $e:=-\ord_p(x) \geq 1$.   We have $\ord_p(d x^4)=-4e$, $\ord_p(5(m+16n^2)x^2)=-2e$ and $\ord_p(4(m+16n^2)(m+25n^2)/d)=0$.  From (\ref{E:Selmer 2a}), we deduce that $\ord_p(y^2)=-4e$ and hence $\ord_p(y)=-2e$.   Multiplying (\ref{E:Selmer 2a}) by $p^{4e}$ gives $(p^{2e} y)^2= d(p^e x)^4-p^{2e}5(m+16n^2)(p^ex)^2 + p^{4e}4(m+16n^2)(m+25n^2)/d$.  Using that $p^{2e}y,p^e x \in \ZZ_p^\times$ and reducing modulo $p$, we find that $d$ is a square modulo $p$ which again is a contradiction.    We conclude that $d$ is a square modulo $p$.

  By Lemma~\ref{L:easy d}, $d$ divides $4(m+16n^2)(m+25n^2)$.  Since $d>0$, we deduce that $d$ lies in the subgroup of $\QQ^\times/(\QQ^\times)^2$ generated by $2$, $m+16n^2$ and $m+25n^2$.   Set $p:=m$.  We have $\legendre{2}{p}=-1$ since $p=m\equiv 3 \pmod{8}$.  We have $\legendre{m+16n^2}{p}=\legendre{16n^2}{p}=1$ and $\legendre{m+25n^2}{p}=\legendre{25n^2}{p}=1$.  Using these Legendre symbols with $d$ being a square modulo $p$, we conclude that $d$ lies in the set $\{1,m+16n^2, m+25n^2, (m+16n^2)(m+25n^2)\}$.
\end{proof}

\subsection{Computation of the weak Mordell--Weil group}

\begin{lemma} \label{L:almost weak MW}
\begin{romanenum}
\item \label{L:almost weak MW i}
The group $E'(\QQ)/\phi(E(\QQ))$ is isomorphic to $(\ZZ/2\ZZ)^2$ and is generated by the points $Q_0$ and $Q_1$.
\item \label{L:almost weak MW ii}
The group $E(\QQ)/\hat\phi(E'(\QQ))$ is isomorphic to $(\ZZ/2\ZZ)^2$ and is generated by the points $P_0$ and $P_1$.
\end{romanenum}
\end{lemma}
\begin{proof}
As noted in \S\ref{S:background}, we have a group homomorphism $\delta\colon E'(\QQ)\to \QQ^\times/(\QQ^\times)^2$ whose kernel equals $\phi(E(\QQ))$.   By Lemma~\ref{L:Selmer 1}, we have inclusions
\begin{align} \label{E:inclusions Sel 1}
\delta(E'(\QQ))\subseteq \Sel_\phi(E/\QQ) \subseteq \{1,-m, -(m+16n^2), m(m+16n^2)\}.
\end{align}
The inclusions of groups in (\ref{E:inclusions Sel 1}) are in fact equalities since $\delta(Q_0)= b'\cdot (\QQ^\times)^2=m(m+16n^2)\cdot (\QQ^\times)^2$ and $\delta(Q_1)=-(m+16n^2)\cdot (\QQ^\times)^2$.   In particular, $\delta(E'(\QQ))$ is isomorphic to $(\ZZ/2\ZZ)^2$ and is generated by $\delta(Q_0)$ and $\delta(Q_1)$.  Part (\ref{L:almost weak MW i}) is now immediate. 

Similarly, we have a group homomorphism $\delta'\colon E(\QQ)\to \QQ^\times/(\QQ^\times)^2$ whose kernel equals $\hat\phi(E'(\QQ))$. By Lemma~\ref{L:Selmer 2}, we have inclusions
\begin{align} \label{E:inclusions Sel 2}
\delta'(E(\QQ))\subseteq \Sel_{\hat\phi}(E'/\QQ) \subseteq \{1, m+16n^2, m+25n^2, (m+16n^2)(m+25n^2)\}.
\end{align}
The inclusions of groups in (\ref{E:inclusions Sel 2}) are in fact equalities since $\delta(P_0)= b\cdot (\QQ^\times)^2=(m+16n^2)(m+25n^2)\cdot (\QQ^\times)^2$ and $\delta(P_1)=(m+16n^2)\cdot (\QQ^\times)^2$.   In particular, $\delta'(E(\QQ))$ is isomorphic to $(\ZZ/2\ZZ)^2$ and is generated by $\delta'(P_0)$ and $\delta'(P_1)$.  Part (\ref{L:almost weak MW ii}) is now immediate. 
\end{proof}

\begin{lemma} \label{L:torsion}
The torsion subgroup of $E(\QQ)$ is the cyclic group of order $2$ generated by $P_0$.
\end{lemma}
\begin{proof}
The Weierstrass model defining our elliptic curve is minimal at $2$ since it has coefficients in $\ZZ$ and its discriminant $\Delta\in \ZZ$ satisfies $\ord_2(\Delta)=8<12$.  Let $\tilde{E}$ be the (singular) curve over $\FF_2$ obtained by reducing this model modulo $2$.  Let $\tilde{E}_{\operatorname{ns}}$ be the open subvariety of $\tilde{E}$ consisting of nonsingular points; it is a commutative group variety by making use of group operations coming from $E$.   Let $E_0(\QQ_2)$ and $E_1(\QQ_2)$ be the subgroups of $E(\QQ_2)$ consisting of those points whose reduction modulo $2$ lies in $\tilde{E}_{\operatorname{ns}}(\FF_2)$ or the identity subgroup of $\tilde{E}_{\operatorname{ns}}(\FF_2)$, respectively.    By \cite[Propositions VII.2.1 and IV.3.2]{Silverman}, any torsion element in $E_1(\QQ_2)$ has order equal to a power of $2$.   The set $\tilde{E}_{\operatorname{ns}}(\FF_2)$ consist of the single point $(1,0)$, so we have $E_1(\QQ_2)=E_0(\QQ_2)$.   

The order of the group $E(\QQ_2)/E_0(\QQ_2)$ can be computed using Tate's algorithm.  Replacing $y$ by $y+x$, we obtain an alternate model for $E$ given by
\[
y^2+2xy = x^3-(5(m+16n^2)-1)x^2+4(m+16n^2)(m+25n^2)x.
\]
With notation as in Tate's algorithm \cite[Algorithm 9.4]{SilvermanII}, we have the following information concerning the basic invariants: $a_1=2$, $a_2=-5(m+16n^2)-1$, $a_3=0$, $a_4=4(m+16n^2)(m+25n^2)$, $a_6=0$, $b_2\equiv 0 \pmod{4}$, $b_4 \equiv 0 \pmod{8}$, $b_6=0$, $b_8\equiv 0 \pmod{16}$.  We have $m\equiv 3 \pmod{4}$ and $n\equiv 0 \pmod{2}$ since $m$ and $m+25n^2$ are both primes that are congruent to $3$ modulo $4$.  Using this, we find that $a_2/2 \equiv 0 \pmod{2}$ and $a_4/4 \equiv 1 \pmod{2}$.  With the above information one can directly check Tate's algorithm \cite[Algorithm 9.4]{SilvermanII} to find that $E$ has Kodaira symbol $\operatorname{I}_r^*$ at $2$ for some $r\geq 1$ and hence $c:=|E(\QQ_2)/E_0(\QQ_2)|$ is $2$ or $4$.

Combining the above results, we deduce that any element of finite order in $E(\QQ_2)$ has order equal to a power of $2$.   This proves that the torsion group $E(\QQ)_{\tors}$ has cardinality a power of $2$.    The only element of order $2$ in $E(\QQ)$ is $P_0=(0,0)$; using the Weierstrass equation of $E$, the other points of order $2$ are defined over the quadratic field $\QQ(\sqrt{m(m+16n^2)})$.  Therefore, $E(\QQ)_{\tors}$ is a cyclic group of order $2^e$ for some $e\geq 1$.   

Suppose that $e\geq 2$ and hence $P_0=2 P$ for some $P\in E(\QQ)$.    We have $P_0= 2 P = \hat\phi(\phi(P)) \in \hat\phi(E'(\QQ))$ which contradicts Lemma~\ref{L:almost weak MW}(\ref{L:almost weak MW ii}).   Therefore, $E(\QQ)_\tors$ is a cyclic group of order $2$ generated by $(0,0)$.
\end{proof}

\begin{lemma} \label{L:WMW}
The group $E(\QQ)/2E(\QQ)$ is isomorphic to $(\ZZ/2\ZZ)^3$ and is generated by the points $P_0$, $P_1$ and $P_2$.
\end{lemma}
\begin{proof}
Using that $\hat\phi\circ \phi = [2]$, we have a short exact sequence
\[
0 \to E'(\QQ)[\hat\phi]/\phi(E(\QQ)[2])\to E'(\QQ)/\phi(E(\QQ)) \xrightarrow{\hat\phi} E(\QQ)/2E(\QQ)\to E(\QQ)/\hat\phi(E'(\QQ)) \to 0.
\]
Using Lemma~\ref{L:torsion} and $\phi(P_0)=0$, we find that the group $E'(\QQ)[\hat\phi]/\phi(E(\QQ)[2])$ has order $2$ and is generated by $Q_0=(0,0)$.   Using Lemma~\ref{L:almost weak MW} and considering the cardinalities of the groups in the short exact sequence, we deduce that $|E(\QQ)/2E(\QQ)|=8$ and hence $E(\QQ)/2E(\QQ)\cong (\ZZ/2\ZZ)^3$.

Using Lemma~\ref{L:almost weak MW} with the short exact sequence, we find that $E(\QQ)/2E(\QQ)$ is generated by $P_0$, $P_1$, $\hat\phi(Q_0)$ and $\hat\phi(Q_1)$.  The lemma follows since $\hat\phi(Q_0)=0$ and $\hat\phi(Q_1)=P_2$.
\end{proof}

Since $E(\QQ)$ is a finitely generated abelian group, Lemmas~\ref{L:torsion} and \ref{L:WMW} imply that $E(\QQ)\cong \ZZ/2\ZZ \times \ZZ^2$.  This completes the proof of Theorem~\ref{T:main examples}.

\section{Proof of Theorem~\ref{T:MAIN}}

The set of primes that are congruent to $11$ modulo $24$ has natural density $1/\varphi(24)=1/8$.  By \cite[Theorem~1.3]{TZ}, there are infinitely many pairs $(m,n)$ of natural numbers for which $m$, $m+16n^2$ and $m+25n^2$ are primes congruent to $11$ modulo $24$.

Fix such a pair $(m,n)$ and let $E/\QQ$ be the elliptic curve defined by the model in Theorem~\ref{T:main examples}.  By Theorem~\ref{T:main examples}, $E(\QQ)$ has rank $2$.   
 
 Let $j_E \in \QQ$ be the $j$-invariant of $E$.  Recall that $j_E$ uniquely determines $E$ up to isomorphism over $\Qbar$.  So to complete the proof of the theorem, it suffices to show that each possible pair $(m,n)$ gives rise to a different $j$-invariant $j_E$.
 
One can verify that 
\[
j_E = \frac{16(13m+100n^2)^3}{3^2m(m+25n^2)^2}.
\]
We now show that this expression for $j_E$ is in lowest terms.
\begin{lemma} \label{L:gcd}
We have $\gcd(16(13m+100n^2)^3,3^2m(m+25n^2)^2)=1$.
\end{lemma}
\begin{proof}
The primes $3$, $m$ and $m+25n^2$ are odd, so we need only verify that they do not divide $13m+100n^2$.   

We have $n\equiv 0\pmod{3}$ since $m$ and $m+16n^2$ are both congruent to $2$ modulo $3$.  Therefore, $13m+100n^2\equiv m\equiv 2\pmod{3}$ and hence $3\nmid (13m+100n^2)$.

Suppose that $m$ divides $13m+100n^2$.  We have $0\equiv 13m+100n^2 \equiv 100n^2 \pmod{m}$ and hence $m$ divides $n$ since $m>5$.   However, this is impossible since $m$ and $m+16n^2$ are both prime.  Therefore, $m \nmid (13m+100n^2)$.

Finally suppose that $p:=m+25n^2$ divides $13m+100n^2$.  We have $0\equiv 13m+100n^2 \equiv 13(-25n^2)+100n^2 = -225 n^2\pmod{p}$ and hence $p$ divides $n$ since $p>5$.   However, this is impossible since $n<p$.  Therefore, $p \nmid (13m+100n^2)$.
\end{proof}

By Lemma~\ref{L:gcd}, the denominator of $j_E$ is $3^2m(m+25n^2)^2$.   Therefore, $m$ and $m+25n^2$ are the two largest primes dividing the denominator of $j_E$.  In particular, we can recover the pair $(m,n)$ from the $j$-invariant of $E$.

\end{document}